\documentclass[11pt, a4paper]{amsart}
\usepackage{amsmath}
\usepackage{amsthm}
\usepackage{amsfonts}
\usepackage{amssymb}
\usepackage{verbatim}
\usepackage{graphicx}
\usepackage{geometry}
\usepackage{a4wide}
\usepackage[latin1]{inputenc}

\DeclareMathOperator{\Mp}{Mp}

\DeclareMathOperator{\Z}{\mathbb{Z}}
\DeclareMathOperator{\R}{\mathbb{R}}
\DeclareMathOperator{\C}{\mathbb{C}}
\DeclareMathOperator{\Q}{\mathbb{Q}}
\renewcommand{\H}{\mathbb{H}}
\DeclareMathOperator{\Aut}{Aut}

\DeclareMathOperator{\e}{\mathfrak{e}}
\DeclareMathOperator{\new}{new}
\DeclareMathOperator{\cusp}{cusp}
\DeclareMathOperator{\Log}{Log}

	\newtheorem{Satz}{Satz}[section]
	\newtheorem{Theorem}[Satz]{Theorem}
	\newtheorem{Lemma}[Satz]{Lemma}
	\newtheorem{Proposition}[Satz]{Proposition} 
	\newtheorem{Corollary}[Satz]{Corollary}
	
	\theoremstyle{definition}

	\newtheorem{Remark}[Satz]{Remark}
\date{\today}

\author{Jan Hendrik Bruinier}
\address{Fachbereich Mathematik, Technische Universit\"at Darmstadt, Schlossgartenstra{\ss}e 7, D--64289 Darmstadt, Germany}
\email{bruinier@mathematik.tu-darmstadt.de}

\author{Markus Schwagenscheidt}
\address{Mathematical Institute, University of Cologne, Weyertal 86-90, D--50931 Cologne, Germany}
\email{mschwage@math.uni-koeln.de}

\title[Converse Theorem]{A converse theorem for Borcherds products on $X_{0}(N)$}

\thanks{Both authors are partially supported by the DFG Research Unit FOR 1920 \lq Symmetry, Geometry and Arithmetic\rq \, and the LOEWE Reseach Unit USAG. The second author is also supported by the SFB-TRR 191 \lq Symplectic Structures in Geometry, Algebra and Dynamics\rq, funded by the DFG}

\begin{document} 

\begin{abstract}
	We show that every Fricke invariant meromorphic modular form for $\Gamma_{0}(N)$ whose divisor on $X_0(N)$ is defined over $\Q$ and supported on Heegner divisors and the cusps is a generalized Borcherds product associated to a harmonic Maass form of weight $1/2$. Further, we derive a criterion for the finiteness of the multiplier systems of generalized Borcherds products in terms of the vanishing of the central derivatives of $L$-function of certain weight~$2$ newforms. We also prove similar results for twisted Borcherds products.
\end{abstract}

\maketitle
	
\section{Introduction}

Borcherds' singular theta correspondence, constructed in his famous paper \cite{borcherds}, yields a multiplicative lifting map from weakly holomorphic modular forms of weight $1-n/2$ for the Weil representation associated to an even lattice $L$ of signature $(2,n)$ to meromorphic modular forms for the orthogonal group of $L$. The modular forms arising in this way have particular infinite product expansions at the cusps and are therefore called Borcherds products. It is an important question, raised by Borcherds in \cite{borcherds}, Problem~16.10, whether there is a converse theorem for Borcherds products, i.e., whether every meromorphic modular form for the orthogonal group of $L$ whose divisor is a linear combination of Heegner divisors can be obtained as a Borcherds product of a weakly holomorphic modular form. If $n \geq 2$, then the best known general result in this direction states that the converse theorem holds if $L$ splits a hyperbolic plane and a rescaled hyperbolic plane over $\Z$, see \cite{bruinierconversetheorem}. We also refer to \cite{bruinierhabil}, Theorem~5.12, and \cite{bruinierfreitag,bruinierfunkeinjectivity,heimmurase} for related results. On the other hand, if $L$ has signature $(2,1)$, then there are counter-examples to the converse theorem. They arise from additional relations between Heegner divisors which are implied by the Gross-Zagier formula, but which cannot be obtained as Borcherds products of weakly holomorphic modular forms. In the present paper we investigate the failure of the converse theorem in signature $(2,1)$ in detail for the group $\Gamma_{0}(N)$. First, we prove a weak converse theorem which states that every meromorphic modular form for $\Gamma_{0}(N)$ whose divisor on $Y_0(N)$ is a linear combination of Heegner divisors and whose cuspidal divisor is defined over $\Q$ and Fricke-invariant is the generalized Borcherds product of a \emph{harmonic Maass form} of weight $1/2$. The generalized Borcherds product of a harmonic Maass form transforms with a multiplier system which may be of infinite order. Our second main result gives a criterion for the finiteness of the multiplier system in terms of the vanishing of the central derivatives of the $L$-functions of certain newforms of weight $2$. Let us describe our results in more detail. 

\subsection{A weak converse theorem} Let $N$ be a positive integer. We let $Y_{0}(N)=\Gamma_0(N)\backslash \H$ be the modular curve corresponding to $\Gamma_{0}(N)$ and denote by $X_0(N)$ its compactification by the cusps.  
By a meromorphic modular form of real weight $\kappa \in \R$ for $\Gamma_{0}(N)$ with a unitary multiplier system $\sigma : \Gamma_{0}(N) \to S^{1}$ we mean a meromorphic function $F: \H \to \C$ which transforms as
\begin{align*}
F\left(Mz\right) = \sigma(M)(cz+d)^{\kappa}F(z)
\end{align*}
for all $z \in \H$ and $M \in \Gamma_{0}(N)$, and which is meromorphic at the cusps. Here we let $z^{a} = \exp(a\Log(z))$ for $z,a \in \C$, where $\Log$ denotes the principal branch of the logarithm. We say that $\sigma$ has finite order if there is a positive integer $m$ such that $\sigma(M)^{m} = 1$ for all $M \in \Gamma_{0}(N)$.

We recall some basic facts about vector valued harmonic Maass forms for the Weil representation from \cite{bruinierfunke04}. Let $L_{N} = \Z(N) \oplus U$ where $\Z(N)$ is the lattice $\Z$ equipped with the quadratic form $Q(x) = Nx^{2}$, and $U$ is the hyperbolic plane $\Z^{2}$ with $Q(x,y) = xy$. It has signature $(2,1)$ and level $4N$. Write $L_{N}'$ for the dual of $L_{N}$. The discriminant group $D_{N} = L_{N}'/L_{N}$ is isomorphic as a finite quadratic module to $\Z/2N\Z$ with the finite quadratic form $Q(\gamma) = \gamma^{2}/4N$, and we will use this identification without further notice. The hermitian symmetric domain associated with $L_{N}$ can be identified with the complex upper half-plane, and modular forms for the discriminant kernel of the orthogonal group of $L_{N}$ can be viewed as classical elliptic modular forms for $\Gamma_{0}(N)$. Let $\C[D_{N}]$ be the group ring of $L_{N}$ with standard basis vectors $\e_{\gamma}$ for $\gamma \in D_{N}$. 
The Weil representation $\rho_{N}$ associated with $L_{N}$ is a unitary representation of the integral metaplectic group $\Mp_{2}(\Z)$ on the group ring $\C[D_{N}]$, see \cite{borcherds}. We let $\overline{\rho}_{N}$ denote the corresponding dual Weil representation. A smooth function $f: \H \to \C[D_{N}]$ is called a harmonic Maass form of weight $k \in \frac{1}{2}+\Z$ for $\rho_{N}$ if it is annihilated by the weight $k$ 
Laplace operator $\Delta_{k}$, if it transforms under $\Mp_2(\Z)$ like a modular form of weight $k$ for $\rho_{N}$, and if it is at most of linear exponential growth at the cusp $\infty$. The antilinear differential operator
\begin{align*}
\xi_{k}f = 2iv^{k}\overline{\frac{\partial}{\partial \overline{\tau}}f} \qquad (\tau = u+iv \in \H) 
\end{align*}
maps harmonic Maass forms of weight $k$ for $\rho_{N}$ to weakly holomorphic modular forms of weight $2-k$ for $\overline{\rho}_{N}$. We let $H_{k,\rho_{N}}^{+}$ be the space of all harmonic Maass forms of weight $k$ for $\rho_{N}$ which are mapped to the space $S_{2-k,\overline{\rho}_{N}}$ of cusp forms of weight $2-k$ for $\overline{\rho}_{N}$, and we let $M_{k,\rho_{N}}^{!}$ be the subspace of weakly holomorphic modular forms. Every $f \in H_{k,\rho_{N}}^{+}$ has a Fourier expansion of the form
\begin{align*}
f(\tau) = \sum_{\gamma \in D_{N}}\left(\sum_{\substack{n \in \Z\\ n \gg -\infty}}a_{f}^{+}(n,\gamma)e\left(\frac{n\tau}{4N}\right)+\sum_{\substack{n \in \Z \\ n < 0}}a_{f}^{-}(n,\gamma)\Gamma\left(1-k,\frac{\pi|n|v}{N}\right)e\left(\frac{n\tau}{4N}\right)\right)\e_{\gamma},
\end{align*}
with $e(z) = e^{2\pi i z}$ for $z \in \C$, coefficients $a_{f}^{\pm}(n,\gamma) \in \C$ and $\Gamma(s,x) = \int_{x}^{\infty}e^{-t}t^{s-1}dt$ the incomplete gamma function. Note that the transformation behaviour under $\rho_{N}$ implies that $a_{f}^{\pm}(n,-\gamma) = (-1)^{k-1/2}a_{f}^{\pm}(n,\gamma)$ for all $n \in \Z, \gamma \in D_{N},$ and $a_{f}^{\pm}(n,\gamma) =0$ unless $n \equiv \gamma^{2}(4N)$. The finite Laurent polynomial
\[
P_{f}(\tau) = \sum_{\gamma \in D_{N}}\sum_{\substack{n \in \Z\\ n \leq 0}}a_{f}^{+}(n,\gamma)e\left(\frac{n\tau}{4N}\right)\e_{\gamma}
\]
is called the principal part of $f$. We emphasize that the principal part includes the coefficients with $n = 0$.

Next, we explain the generalized Borcherds products associated with harmonic Maass forms, as  defined in \cite{bruinierono}, Section 6. Let $f \in H_{1/2,\rho_{N}}^{+}$ and assume that the coefficients $a_{f}^{+}(n^{2},n)$ are real for all $n$, and that the coefficients $a_{f}^{+}(n,\gamma)$ for $n < 0, \gamma \in D_{N},$ are integers. Then by Theorem~6.1 of \cite{bruinierono} the infinite product
\[
\Psi(f,z) = e(\rho_{f}z)\prod_{n = 1}^{\infty}(1-e(nz))^{a_{f}^{+}(n^{2},n)}
\]
converges for $\Im(z) \gg 0$ large enough and has a continuation to all of $\H$ which is a meromorphic modular form of weight $a_{f}^{+}(0,0)$ for $\Gamma_{0}(N)$ with a unitary multiplier system (possibly of infinite order). Here $\rho_{f} \in \R$ is the Weyl vector at the cusp $\infty$ associated with $f$, which can be computed in terms of the Petersson inner product of $f$ with a unary theta series of weight $1/2$. 
The divisor on $Y_0(N)$ of the generalized Borcherds product $\Psi(f,z)$ is given in terms of the principal part of $f$ by the linear combination
\[
Z(f) = \sum_{\gamma \in D_{N}}\sum_{n < 0}a_{f}^{+}(n,\gamma)Z(n,\gamma),
\]
where $Z(n,\gamma)$ is the Heegner divisor defined in \cite{bruinierono}, Section 5. The cuspidal part of the divisor of $f$ can explicitly be described in terms of the Weyl vectors   associated to $f$ at the different cusps. What is important to us is the fact that if we represent the cusps of $\Gamma_{0}(N)$ by fractions $\frac{a}{c}$ with $a,c \in \Z_{>0}, (a,c)  =1$ and $c \mid N$, then the order of $\Psi(f,z)$ at $\frac{a}{c}$ only depends on $c$, but not on $a$ (compare \cite{bruinierschwagenscheidt}, Section~5). It is well known that the cusp $\frac{a}{c}$ is defined over the cyclotomic field $\Q(\zeta_{(c,N/c)})$ and that the cusps $\frac{a}{c}$ with fixed $c$ form a complete Galois orbit. Consequently, the cuspidal part of the divisor of $\Psi(f,z)$ is defined over $\Q$. Furthermore, the divisor of $\Psi(f,z)$ is invariant under the Fricke involution $W_{N} = \left(\begin{smallmatrix}0 & -1 \\ N & 0 \end{smallmatrix}\right)$. Our first main result is the following weak converse theorem.

\begin{Theorem}\label{weak converse theorem}
	Let $F$ be a meromorphic modular form for $\Gamma_{0}(N)$ with a unitary multiplier system (possibly of infinite order). Suppose that the divisor of $F$ on $Y_0(N)$ is a linear combination of Heegner divisors $Z(n,\gamma)$ and that its divisor at the cusps is defined over $\Q$ and invariant under the Fricke involution $W_{N}$. Then $F$ is (up to a non-zero constant factor) the generalized Borcherds product associated to a unique harmonic Maass form $f \in H_{1/2,\rho_{N}}^{+}$.
\end{Theorem}

\begin{Remark}
	The supplement \lq weak\rq \, indicates that $f$ is a harmonic Maass form, rather than a weakly holomorphic modular form.
\end{Remark}

It is well known that any linear combination of Heegner divisors $Z(n,\gamma)$ can be realized as the divisor on $Y_0(N)$ of the generalized Borcherds product associated to a suitable harmonic Maass form $f \in H_{1/2,\rho_{N}}^{+}$ of weight $1/2$. The main point of the above theorem is that we can choose $f$ in such a way that its associated generalized Borcherds product also has the correct cuspidal divisor. We construct this $f$ by considering the Borcherds products of suitable unary theta series of weight $1/2$. We refer to Section~\ref{section weak converse theorem} for a detailed proof of the theorem. 

\subsection{Generalized Borcherds products with multiplier systems of finite order} Next, we want to describe those harmonic Maass forms $f$ whose associated generalized Borcherds products $\Psi(f,z)$ transform with multiplier systems of finite order. Using transcendence results of Waldschmidt, W\"ustholz and Scholl  (see e.g.~\cite{waldschmidt}, \cite{wuestholz}, \cite{scholl}) for periods of differentials of the third kind, it was proved in \cite{bruinierono} that the finiteness of the multiplier system is equivalent to the rationality of certain coefficients of the holomorphic part of $f$, see Theorem~\ref{theorem bruinierono1} below. Furthermore, if we are in the special situation that $g = \xi_{1/2}f \in S_{3/2,\overline{\rho}_{N}}$ is a newform, then an application of the Gross-Zagier formula shows that the rationality of these coefficients of $f$ is equivalent to the vanishing of the central derivative $L'(G,1)$ of the $L$-function of the weight $2$ newform $G$ for $\Gamma_{0}(N)$ corresponding to $g$ under the Shimura correspondence, see Theorem~\ref{theorem bruinierono2} below. We will give a similar criterion for the finiteness of the multiplier system of the generalized Borcherds product associated to an arbitrary harmonic Maass form $f \in H_{1/2,\rho_{N}}^{+}$, which does not necessarily map to a newform under $\xi_{1/2}$.

The space $S_{3/2,\overline{\rho}_{N}}$ is isomorphic to the space $J_{2,N}^{\cusp}$ of Jacobi cusp forms of weight $2$ and index $N$, see \cite{eichlerzagier}, Theorem~5.1. There is an extensive theory of Hecke operators and newforms for Jacobi forms which carries over to vector valued modular forms. In particular, there is a Hecke operator $T_{n}$ acting on $S_{3/2,\overline{\rho}_{N}}$ for every positive integer $n$, and the space $S_{3/2,\overline{\rho}_{N}}^{\new}$ of newforms for $\overline{\rho}_{N}$ is defined. By the Shimura correspondence it is isomorphic as a module over the Hecke algebra to the space $S_{2}^{\new,-}(N)$ of newforms of weight $2$ for $\Gamma_{0}(N)$ on which the Fricke involution acts with eigenvalue $+1$. The results of \cite{skoruppazagier} yield a direct sum decomposition
\begin{align}\label{eq direct sum decomposition}
S_{3/2,\overline{\rho}_{N}} = \bigoplus_{\substack{d,\ell > 0 \\ d^{2}\ell \mid N}}S_{3/2,\overline{\rho}_{N/d^{2}\ell}}^{\new}|U_{d}V_{\ell}
\end{align}
where $U_{d}: S_{3/2,\overline{\rho}_{N/d^{2}}} \to S_{3/2,\overline{\rho}_{N}}$ and $V_{\ell}: S_{3/2,\overline{\rho}_{N/\ell}} \to S_{3/2,\overline{\rho}_{N}}$ are the translations to vector valued modular forms of the usual index raising operators on Jacobi forms, compare Section~\ref{section weak converse theorem} below. If $G \in S_{2}^{\new,-}(N/m)$ for some $m \mid N$ is a newform and $g \in S_{3/2,\overline{\rho}_{N/m}}^{\new}$ is the newform corresponding to $G$ under the Shimura correspondence, then the cusp forms $g|U_{d}V_{\ell}$ with $d^{2}\ell = m$ form a basis for the space of those forms in $S_{3/2,\overline{\rho}_{N}}$ which have the same eigenvalues as $G$ under all Hecke operators $T_{n}$ with $(n,N) = 1$. We call this space the $G$-isotypical component of $S_{3/2,\overline{\rho}_{N}}$.

By subtracting a suitable multiple of a unary theta function from $f$ we can assume that $a_{f}^{+}(0,0) =0$. Then the associated generalized Borcherds product $\Psi(f,z)$ has weight $0$ and its divisor has degree $0$. Let $J$ be the Jacobian of $X_{0}(N)$. For a number field $F$, we let $J(F)$ be its points over $F$. They correspond to divisors of degree $0$ on $X_{0}(N)$ which are defined over $F$. We define a degree $0$ divisor corresponding to $Z(f)$ by putting
\[
y(f) = Z(f)-\deg(Z(f))\cdot\infty \in J(\Q).
\]
Note that $y(f)$ and the divisor of the generalized Borcherds product $\Psi(f,z)$ differ by a degree $0$ divisor supported at the cusps. By the Manin-Drinfeld theorem, $y(f)$ and the divisor of $\Psi(f,z)$ define the same point in $J(\Q) \otimes \R$. Our second main result  is the following criterion for the finiteness of the multiplier systems of generalized Borcherds products.


\begin{Theorem}\label{finiteness theorem}
	Let $f \in H_{1/2,\rho_{N}}^{+}$ be a harmonic Maass form with real coefficients $a_{f}^{+}(n,\gamma)$ for $n \in \Z,\gamma \in D_{N}$, and integral coefficients $a_{f}^{+}(n,\gamma)$ for $n < 0, \gamma \in D_{N}$. Further assume that $a_{f}^{+}(0,0) = 0$, $f$ is orthogonal to cusp forms with respect to the regularized Petersson inner product, and the principal part of $f$ is defined over $\Q$. Then the following statements are equivalent.
	\begin{enumerate}
		\item The multiplier system of the Borcherds product $\Psi(f,z)$ has finite order.
		\item The divisor $y(f)$ is torsion in the Jacobian of $X_{0}(N)$.
		\item The coefficients $a_{f}^{+}(n^{2},n)$ are rational for all $n \in \Z$.
		\item The coefficients $a_{f}^{+}(n^{2},n)$ are algebraic for all $n \in \Z$.
		\item We have
		\[
		\left(\xi_{1/2}f,g\right)L'(G,1)= 0
		\]
		for every newform $G \in S_{2}^{\new,-}(N/m)$ for $m\mid N$ and every cusp form $g$ in the $G$-isotypical component of $S_{3/2,\overline{\rho}_{N}}$. Here $(\cdot,\cdot)$ denotes the Petersson inner product on $S_{3/2,\overline{\rho}_{N}}$.
		\end{enumerate}
\end{Theorem}

\begin{Remark}
	\begin{enumerate}
		\item It is easy to see that we can write every $f \in H_{1/2,\rho_{N}}^{+}$ with real cofficients in the holomorphic part and integral coefficients $a_{f}^{+}(\gamma,n)$ for $n < 0,\gamma \in D_{N},$ as $f = f_{1} + f_{2}$ with $f_{1} \in H_{1/2,\rho_{N}}^{+}$ satisfying the conditions required in the theorem, and $f_{2} \in M_{1/2,\rho_{N}}$ a holomorphic modular form. From the facts that $M_{1/2,\rho_{N}}$ has a basis consisting of unary theta series (see Lemma~\ref{lemma basis unary theta series}) whose Borcherds products are given by explicit eta products (see Corollary~\ref{corollary eta products}) we see that the Borcherds product $\Psi(f_{2},z)$ transforms with a multiplier system of finite order if and only if all of the coefficients of $f_{2}$ are rational. Combining this with the theorem we obtain a criterion for the finiteness of the multiplier system of the generalized Borcherds product $\Psi(f,z)$, without the additional assumptions on $f$.
		\item The equivalence of (1) and (2) is clear by the remarks preceeding the theorem, and the equivalence of (2), (3) and (4) is the statement of Theorem~6.2 in \cite{bruinierono} (which we recall in Theorem~\ref{theorem bruinierono1} below). We included all these equivalent statements for the convenience of the reader. The main point of the theorem is the equivalence of (1) and (5), which will be proved in Section~\ref{section finiteness theorem}.
		\item By the results of \cite{skoruppazagier}, the space $S_{3/2,\overline{\rho}_{N}}\cong J_{2,N}^{\cusp}$ is isomorphic as a module over the Hecke algebra to the cuspidal part of a certain subspace $\mathcal{M}_{2}^{-}(N)$ of the space of holomorphic modular forms of weight $2$ for $\Gamma_{0}(N)$ which have eigenvalue $+1$ under the Fricke involution $W_{N}$. If we denote this isomorphism by $\varphi_{N}: S_{3/2,\overline{\rho}_{N}} \stackrel{\sim}{\to} \mathcal{M}_{2}^{-}(N)$, then the above theorem implies that if $\Psi(f,z)$ transforms with a multiplier system of finite order, then
		\[
		L'(\varphi_{N}(\xi_{1/2}f),1) = 0.
		\] 
		However, the vanishing of this $L$-derivative is not equivalent to item (5) in the above theorem.
	\end{enumerate}
\end{Remark}

\subsection{A converse theorem for twisted Borcherds products}

In \cite{bruinierono}, Theorem 6.1, the authors also defined certain twisted generalized Borcherds products $\Psi_{\Delta,r}(f,z)$ of harmonic Maass forms $f \in H_{1/2,\tilde{\rho}_{N}}^{+}$, where $\Delta$ is a fundamental discriminant and $r \in \Z/2N\Z$ satisfies $r^{2} \equiv \Delta (4N)$, and $\tilde{\rho}_{N}$ is either $\rho_{N}$ or $\overline{\rho}_{N}$, depending on whether $\Delta > 0$ or $\Delta < 0$. For brevity, we do not repeat the definition of these twisted Borcherds products here but refer the reader to \cite{bruinierono}. If $\Delta \neq 1$, then $\Psi_{\Delta,r}(f,z)$ is a meromorphic modular form of weight $0$ for $\Gamma_{0}(N)$ with a unitary multiplier system (possibly of infinite order). The divisor of $\Psi_{\Delta,r}(f,z)$ is supported on $Y_0(N)$ and is given by the linear combination
\[
Z_{\Delta,r}(f) = \sum_{\gamma \in D_{N}}\sum_{n < 0}a_{f}^{+}(n,\gamma)Z_{\Delta,r}(n,\gamma)
\]
with the twisted Heegner divisors $Z_{\Delta,r}(n,\gamma)$ as defined in \cite{bruinierono}, Section~5. For $\Delta \neq 1$ the divisor $Z_{\Delta,r}(n,\gamma)$ has degree $0$, therefore we set $y_{\Delta,r}(f) = Z_{\Delta,r}(f)$ in this case.

The space $S_{3/2,\overline{\tilde{\rho}}_{N}}^{\new}$ is isomorphic to the space of (holomorphic if $\Delta > 0$, skew-holomorphic if $\Delta < 0$) cuspidal Jacobi newforms of weight $2$ and index $N$ , which is in turn isomorphic as a module over the Hecke algebra to the space $S_{2}^{\new,\mp}(N)$ ($\mp = -$ if $\Delta > 0$, $\mp = +$ if $\Delta < 0$) of newforms of weight $2$ for $\Gamma_{0}(N)$ on which the Fricke involution acts with eigenvalue $\pm 1$. Here we understand that for $\Delta < 0$ the space $S_{3/2,\rho_{N}}^{\new}$ consists of those newforms which are orthogonal with respect to the Petersson inner product to the space of unary theta series of weight $3/2$ (as defined in \cite{bruinierschwagenscheidt}), 
since the latter functions correspond to Eisenstein series under the Shimura correspondence. We have the following twisted version of Theorem~\ref{weak converse theorem} and Theorem~\ref{finiteness theorem}.
 
\begin{Theorem}
	Let $\Delta \neq 1$ and let $r \in \Z/2N\Z$ with $r^{2}\equiv \Delta (4N)$. Let $F$ be a meromorphic modular form for $\Gamma_{0}(N)$ with a unitary multiplier system (possibly of infinite order). Suppose that the divisor of $F$ on $X_0(N)$ is a linear combination of twisted Heegner divisors $Z_{\Delta,r}(n,\gamma)$. There exists a harmonic Maass form $f \in H_{1/2,\tilde{\rho}_{N}}^{+}$ whose $(\Delta,r)$-twisted generalized Borcherds product $\Psi_{\Delta,r}(f,z)$ is a non-zero constant multiple of $F$. Further, the following statements are equivalent.
	\begin{enumerate}
		\item The multiplier system of $F$ has finite order.
		\item The divisor $y_{\Delta,r}(f)$ is torsion in the Jacobian of $X_{0}(N)$.
		\item The coefficients $a_{f}^{+}(|\Delta|n^{2},rn)$ are rational for all $n \in \Z$.
		\item The coefficients $a_{f}^{+}(|\Delta|n^{2},rn)$ are algebraic for all $n \in \Z$.
		\item We have
		\[
		\left(\xi_{1/2}f,g\right)L'(G,\chi_{\Delta},1)= 0
		\]
		for every newform $G \in S_{2}^{\new,\pm}(N/m)$ for $m\mid N$ and every cusp form $g$ in the $G$-isotypical component of $S_{3/2,\overline{\tilde{\rho}}_{N}}$. Here $\chi_{\Delta} = \left( \frac{\Delta}{\cdot}\right)$ is the Kronecker symbol.
		\end{enumerate}
\end{Theorem}

\begin{Remark}
	Note that $M_{1/2,\rho_{N}}$ has a basis consisting of theta series (see Lemma~\ref{lemma basis unary theta series} below) and $M_{1/2,\overline{\rho}_{N}} \cong J_{1,N} = \{0\}$ by a well-known result of Skoruppa (compare \cite{eichlerzagier}, Theorem~5.7). Using this it follows immediately from the definition of the twisted Borcherds product that $\Psi_{\Delta,r}(f,z) = 1$ for $\Delta \neq 1$ and $f \in M_{1/2,\tilde{\rho}_{N}}$. Hence, in contrast to the the untwisted case, for $\Delta \neq 1$ we do not need to assume that $f$ is orthogonal to cusp forms or that its coefficients $\sum_{\gamma \in D_{N}}a_{f}^{+}(0,\gamma)\e_{\gamma}$ are rational. Further, this shows that the harmonic Maass form $f$ in the theorem is only unique up to addition of a holomorphic modular form.
\end{Remark}

The proof of the theorem is analogous to the proofs of Theorem~\ref{weak converse theorem} and Theorem~\ref{finiteness theorem}, so we will leave the details to the reader. 
We only mention that if $\Delta < 0$ and $\xi_{1/2}f$ lies in the space of unary theta series of weight $3/2$, then the holomorphic part of $f$ has rational Fourier coefficients by \cite{bruinierschwagenscheidt}, Theorem~4.7. Hence the above converse theorem is trivially true for such $f$.

The work is organized as follows. In Section~\ref{section operators} we translate some classical operators on Jacobi forms into the vector valued setting. In Section~\ref{section results bruinier ono} we recall some results of Ono and the first named author \cite{bruinierono} on the algebraicity of Fourier coefficients of harmonic Maass forms, which are crucial for the proofs of our main theorems. In Section~\ref{section weak converse theorem} we prove the weak converse theorem, Theorem~\ref{weak converse theorem}, and in Section~\ref{section finiteness theorem} we prove the finiteness criterion for generalized Borcherds products, Theorem~\ref{finiteness theorem}. Along the way we obtain some general results about newforms and harmonic Maass forms of half-integral weight which might be of independent interest. 

\section{Operators on vector valued modular forms}\label{section operators}

We recall some classical operators on Jacobi forms from \cite{eichlerzagier}, viewed here as operators on vector valued modular forms. Let $\rho$ be one of the representations $\rho_{N}$ or $\overline{\rho}_{N}$, and let $k \in \frac{1}{2}+\Z$.
	
	 The automorphism group $\Aut(D_{N})$ acts on vector valued modular forms $f = \sum_{\gamma}f_{\gamma}\e_{\gamma}$ for $\rho$ by $f^{\sigma} = \sum_{\gamma}f_{\gamma}\e_{\sigma(\gamma)}$.	 The elements of $\Aut(D_{N})$ are all involutions, also called Atkin-Lehner involutions, and correspond to the exact divisors $c\mid \mid N$ (i.e., $c\mid N$ and $(c,N/c) =1$). The automorphism $\sigma_{c}$ corresponding to $c$ is defined by the equations
	\begin{align*}
	\sigma_{c}(\gamma) \equiv -\gamma \ (2c) \quad \text{and} \quad \sigma_{c}(\gamma) \equiv \gamma \ (2N/c)
	\end{align*}
	for $\gamma \in D_{N}$, compare \cite{eichlerzagier}, Theorem~5.2.
	
	For each positive integer $n$ there is a Hecke operator $T_{n}$ acting on the space $M_{k,\rho}^{!}$. For example, if $n = p$ is a prime with $(p,N) = 1$, then the action of $T_{p}$ on the Fourier expansion of a weakly holomorphic vector valued modular form 
\[
f = \sum_{\gamma \in D_{N}}\sum_{n \gg -\infty}a_{f}(n,\gamma)e\left(\frac{n\tau}{4N} \right)\e_{\gamma} \in M_{1/2,\rho}^{!}
\]
of weight $k$ for $\rho$ is given by
\[
f|T_{p} = \sum_{\gamma \in L'/L}\sum_{n \gg -\infty}a_{f|T_{p}}(n,\gamma)e\left(\frac{n\tau}{4N} \right)\e_{\gamma} \in M_{1/2,\rho}^{!}
\]
with
\[
a_{f|T_{p}}(n,\gamma) = a_{f}(p^{2}n,p\gamma) + p^{k-3/2}\left(\frac{\sigma n}{p}\right)a_{f}(n,\gamma) + p^{2k-2}a_{f}(n/p^{2},\gamma/p),
\]
where $\sigma = 1$ if $\rho = \rho_{N}$ and $\sigma = -1$ if $\rho = \overline{\rho}_{N}$. 

There are level raising operators $U_{d}$ and $V_{\ell}$ which act on the Fourier expansion of a weakly holomorphic modular form of weight $k$ for $\overline{\rho}_{N}$ by
\begin{align*}
f|U_{d} &= \sum_{\gamma \in D_{Nd^{2}}}\sum_{n \gg -\infty}a_{f}(n/d^{2},\gamma/d)e\left(\frac{n\tau}{4Nd^{2}}\right)\e_{\gamma} \in M_{k,\overline{\rho}_{Nd^{2}}}^{!}, \\
f|V_{\ell} &= \sum_{\gamma \in D_{N\ell}}\sum_{n \gg -\infty}\left(\sum_{a \mid (\frac{n+\gamma^{2}}{4N},\gamma,\ell)}a^{k-1/2}a_{f}(n/a^{2},\gamma/a) \right)e\left(\frac{n\tau}{4N\ell}\right)\e_{\gamma} \in M_{k,\overline{\rho}_{N\ell}}^{!}
\end{align*}
Their actions on $M_{k,\rho_{N}}^{!}$ almost look the same, but in the action of $V_{\ell}$ we have to replace $n$ with $-n$ in the greatest common divisor in the index of the innermost sum. These two operators commute with each other for all $d$ and $\ell$, and they commute with the Hecke operators $T_{n}$ for $(n,N) = 1$.

All the above operators act on harmonic Maass forms as well, their actions on the Fourier expansion being the same. For $f \in H_{k,\rho}^{+}$ we have the following commutation relations with the $\xi$-operator,
\begin{align*}
	\xi_{k}(f^{\sigma_{c}}) & = (\xi_{k}f)^{\sigma_{c}}, \\
	\xi_{k}(f|T_{n}) &= n^{2k-2}(\xi_{k}f)|T_{n}, \\
	\xi_{k}(f|U_{d}) &= (\xi_{k}f)|U_{d}, \\
	\xi_{k}(f|V_{\ell}) &= \ell^{k-1}(\xi_{k}f)|V_{\ell},
\end{align*}
which can be verified using the explicit actions of the operators on the Fourier expansions.


\section{Algebraicity results for the coefficients of harmonic Maass forms}\label{section results bruinier ono}

We recall a few results from \cite{bruinierono} which relate the finiteness of the multiplier system of the generalized Borcherds product of $f \in H_{1/2,\rho_{N}}^{+}$ with the algebraicity of some of the Fourier coefficients of $f$ and the vanishing of certain central $L$-derivatives of weight $2$ newforms.

\begin{Theorem}[\cite{bruinierono}, Theorem 6.2]\label{theorem bruinierono1}
	Let $f \in H_{1/2,\rho_{N}}^{+}$ be a harmonic Maass form with real coefficients $a_{f}^{+}(n,\gamma)$ for all $n \in \Z,\gamma \in D_{N}$, and integral coefficients $a_{f}^{+}(n,\gamma) \in \Z$ for all $n < 0, \gamma \in D_{N}$. Further assume that $a_{f}^{+}(0,0) = 0$, $f$ is orthogonal to cusp forms, and the principal part of $f$ is defined over $\Q$. Then the following statements are equivalent.
	\begin{enumerate}
		\item The multiplier system of the Borcherds product $\Psi(f,z)$ has finite order.
		\item The coefficients $a_{f}^{+}(n^{2},n)$ are rational for all $n \in \Z$.
		\item The coefficients $a_{f}^{+}(n^{2},n)$ are algebraic for all $n \in \Z$.
	\end{enumerate}
\end{Theorem}

We further require Theorem~7.6 and Theorem~7.8 from \cite{bruinierono}. The theorems are only stated in twisted versions for $\Delta \neq 1$ in the reference. The following proposition is needed to obtain an untwisted version.

\begin{Proposition}\label{proposition rational basis}
	The orthogonal complement of $S_{1/2,\rho_{N}}$ in $M_{1/2,\rho_{N}}^{!}$ has a basis consisting of weakly holomorphic modular forms with rational Fourier coefficients. In particular, if $f \in M_{1/2,\rho_{N}}^{!}$ is a weakly holomorphic modular form whose principal part is defined over some number field $K$ and which is orthogonal to cusp forms, then all coefficients of $f$ lie in $K$.
\end{Proposition}

\begin{proof}
	We first show the second claim. For every $m >0, \beta\in D_{N},$ with $m \equiv \beta^{2}(4N)$ there exists a harmonic Maass form $\mathcal{P}_{m,\beta}\in H_{3/2,\overline{\rho}_{N}}^{+}$  with principal part $e(-m\tau/4N)(\e_{\beta}+\e_{-\beta}) + \mathfrak{c}$ for some constant $ \mathfrak{c} \in \C[D_{N}]$, compare \cite{bruinierfunke04}, Proposition 3.11. By adding a suitable linear combination of Eisenstein series of weight $3/2$ we can assume that $\mathfrak{c} \in \Q[D_{N}]$, so the principal part of $\mathcal{P}_{m,\beta}$ is defined over $\Q$. Since $\xi_{3/2}\mathcal{P}_{m,\beta} \in S_{1/2,\rho_{N}}$ and the latter space has a basis consisting of unary theta series (see Lemma \ref{lemma basis unary theta series} below), Theorem~4.7 in \cite{bruinierschwagenscheidt} shows that we can assume that all coefficients $a_{\mathcal{P}_{m,\beta}}^{+}(n,\gamma), n \geq 0, \gamma \in D_{N},$ are rational by adding a suitable cusp form. Using Stokes' theorem (compare \cite{bruinierfunke04}, Proposition~3.5), we obtain
	\[
	0 = (f,\xi_{3/2}\mathcal{P}_{m,\beta}) = 2a_{f}(m,\beta) + \sum_{\gamma \in D_{N}}\sum_{n \geq 0}a_{f}(-n,\gamma)a_{\mathcal{P}_{m,\beta}}^{+}(n,\gamma).
	\]
	Since $a_{f}(-n,\gamma) \in K$ and $a_{\mathcal{P}_{m,\beta}}^{+}(n,\gamma) \in \Q$ for $n \geq 0, \gamma \in D_{N}$, this shows that all coefficients of $f$ lie in $K$.
	
	Now let $\{f_{j}\}$ be a basis of $M_{1/2,\rho_{N}}^{!}$ with rational coefficients. Consider those $f_{j}$ which are not cusp forms. By subtracting from every $f_{j}$ a suitable cusp form we obtain a weakly holomorphic modular form $\tilde{f}_{j}$ which is orthogonal to cusp forms but has the same principal part as $f_{j}$. It is clear that these $\tilde{f}_{j}$ form a basis of the orthogonal complement of $S_{1/2,\rho_{N}}$ in $M_{1/2,\rho_{N}}^{!}$. Since the $\tilde{f}_{j}$ are orthogonal to cusp forms and have rational principal part, they have rational coefficients.
\end{proof}

\begin{Remark}
	The analogous result for the space $M_{1/2,\overline{\rho}_{N}}^{!}$ follows from the fact that $S_{1/2,\overline{\rho}_{N}} \cong J_{1,N}^{\cusp} = \{0\}$ by a result Skoruppa, compare \cite{eichlerzagier}, Theorem~5.7, and that $M_{1/2,\overline{\rho}_{N}}^{!}$ has a basis consisting of forms with integral coefficients by results of McGraw \cite{mcgraw}.
	\end{Remark}
	
	Combined with the above proposition, the remarks after Theorem~7.6 and Theorem~7.8 in \cite{bruinierono} yield the following result.

\begin{Theorem}[\cite{bruinierono}, Theorems 7.6 and 7.8]\label{theorem bruinierono2}
	Let $G \in S_{2}^{\new,-}(N)$ be a newform and let $K_{G}$ be the number field generated by its coefficient. Let $g \in S_{3/2,\overline{\rho}_{N}}^{\new}$ be the newform corresponding to $G$ under the Shimura correspondence, normalized to have coefficients in $K_{G}$. There exists a harmonic Maass form $f \in H_{1/2,\rho_{N}}^{+}$ with $a_{f}^{+}(0,0) = 0$ which is orthogonal to cusp forms and has principal part defined over $K_{G}$ such that $\xi_{1/2}(f) = \|g\|^{-2}g$. Further, the following statements are equivalent.
	\begin{enumerate}
		\item $a_{f}^{+}(n^{2},n)$ is algebraic for all $n \in \Z$.
		\item $a_{f}^{+}(1,1)$ is algebraic.
		\item $L'(G,1) = 0$.
	\end{enumerate}
\end{Theorem} 

\section{Proof of Theorem \ref{weak converse theorem}}\label{section weak converse theorem}

	We consider the unary theta series
	\[
	\theta_{1/2,N}(\tau) = \sum_{\gamma \in D_{N}}\sum_{\substack{n \in \Z \\ n \equiv \gamma(2N)}}e\left(\frac{n^{2}\tau}{4N} \right)\e_{\gamma} \in M_{1/2,\rho_{N}}.
	\]	
	We have the following analog of the Serre-Stark basis theorem for vector valued modular forms. 
	\begin{Lemma}[\cite{bruinierschwagenscheidt}, Lemma~2.1]\label{lemma basis unary theta series}
		Let $\mathcal{D}(N)$ be the set of all positive divisors of $N$ modulo the equivalence relation $d \sim N/d$. Then the theta series
		\[
		\theta_{1/2,N/(d,N/d)^{2}}^{\sigma_{d/(d,N/d)}}|U_{(d,N/d)}, \quad d \in \mathcal{D}(N),
		\]
		form a basis of $M_{1/2,\rho_{N}}$.
	\end{Lemma}
	
	We remark that the above basis was already given in \cite{skoruppazagier}, p. 130, in the language of Jacobi forms. Next, we determine the Borcherds products associated with the above theta series.
	 
	 \begin{Proposition}
	 	\begin{enumerate}
			\item The Borcherds product of $\theta_{1/2,N}^{\sigma_{c}}$ for $c \mid \mid N$ is given by
			\[
			\Psi\left(\theta_{1/2,N}^{\sigma_{c}},z\right) = \eta(cz)\eta\left(\frac{N}{c}z\right).
			\]
			\item For $f \in H_{1/2,\rho_{N}}^{+}$ the Borcherds product of $f|U_{d}$ is given by
			\[
			\Psi\left(f|U_{d},z\right) = \Psi(f,dz).
			\]
		\end{enumerate}
	 \end{Proposition}
	 
	 \begin{proof}
	 	The result follows easily from the explicit actions of $\sigma_{c}$ and $U_{d}$ on Fourier expansions and the definition of the Borcherds product.
	 \end{proof}
	
	\begin{Corollary}\label{corollary eta products}
	For every divisor $d \mid N$ we have
	\[
	\Psi\left(\theta_{1/2,N/(d,N/d)^{2}}^{\sigma_{d/(d,N/d)}}|U_{(d,N/d)},z\right) = \eta(dz)\eta\left(\frac{N}{d}z\right).
	\]
	\end{Corollary}
	
	\begin{Lemma}
		For every finite Laurent polynomial 
		\[
		P(\tau) = \sum_{\gamma \in D_{N}}\sum_{n < 0}a^{+}(n,\gamma)e\left(\frac{n\tau}{4N}\right)\e_{\gamma}
		\]
		with $a^{+}(n,\gamma) = a^{+}(n,-\gamma) \in \C$ and $a^{+}(n,\gamma) = 0$ unless $n \equiv \gamma^{2}(4N)$ there exists a harmonic Maass form $f \in H_{1/2,\rho_{N}}^{+}$ with principal part $P_{f} = P + \mathfrak{c}$ for some constant $\mathfrak{c} \in \C[D_{N}]$. If the coefficients $a^{+}(n,\gamma)$ for $n < 0, \gamma \in D_{N}$, are real, then we can choose $f$ such that all its coefficients $a_{f}^{\pm}(n,\gamma)$ are real for $n \in \Z, \gamma \in D_{N}$.
	\end{Lemma}
	
	
	\begin{proof}
		By \cite{bruinierfunke04}, Proposition~3.11, there exists a harmonic Maass form $f \in H_{1/2,\rho_{N}}^{+}$ with the claimed principal part. Let $f^{c}$ be the function obtained by replacing all the coefficients $a_{f}^{\pm}(n,\gamma)$ of $f$ with their complex conjugates. It is straightforward to check that $f^{c} \in H_{1/2,\rho_{N}}^{+}$ is again a harmonic Maass form. Further, if the coefficients of negative index in the holomorphic part of $f$ are real, then $\frac{1}{2}(f+f^{c})$ has the same coefficients of negative index in the holomorphic part as $f$, and all its coefficients are real. 		
	\end{proof}
	
	\textit{Proof of Theorem \ref{weak converse theorem}.} Let $F$ be a meromorphic modular form for $\Gamma_{0}(N)$ whose divisor on $Y_0(N)$ is a linear combination of Heegner divisors $Z(n,\gamma)$ and whose cuspidal divisor is defined over $\Q$ and invariant under the Fricke involution. By the last lemma we can find a harmonic Maass form $f_{1} \in H_{1/2,\rho_{N}}^{+}$ (unique up to addition of a holomorphic modular form in $M_{1/2,\rho_{N}}$) whose generalized Borcherds product $\Psi(f_{1},z)$ has the same divisor on $Y_0(N)$ as $F$. Hence the divisor of $\Psi(f_{1},z)/F(z)$ is supported at the cusps, defined over $\Q$, and Fricke-invariant. Since the order of a cusp $\frac{a}{c}$ with $c \mid N$ in such a divisor only depends on $c$, and the Fricke involution identifies $\frac{1}{c}$ with $\frac{1}{N/c}$, the $\Q$-vector space consisting of all such divisors has dimension
	\[
	\frac{1}{2}\left(\sigma_{0}(N) +\delta_{N = \square}\right),
	\] 
	where $\sigma_{0}(N)$ is the number of divisors of $N$ and $\delta_{N = \square}$ equals $1$ if $N$ is a square and $0$ otherwise. This is exactly the dimension of the $\Q$-vector space of divisors of the eta products $\eta(dz)\eta(Nz/d)$ for $d \mid N$, which arise as Borcherds products of unary theta series by Corollary~\ref{corollary eta products}. In particular, there exists a unique holomorphic modular form $f_{2} \in M_{1/2,\rho_{N}}$ such that $\Psi(f_{2},z)$ has the same divisor as $\Psi(f_{1},z)/F(z)$. Let $f = f_{1} - f_{2}$. Then the quotient $\Psi(f,z)/F(z)$ is a holomorphic modular form of some weight $k \in \R$ for $\Gamma_{0}(N)$ which does not have any zeros in $\H$ or at the cusps. Hence the function
	\[
	G(z) = 2\log \left(\Im(z)^{k/2}\left|\frac{\Psi(f,z)}{F(z)} \right|\right)
	\]
	is invariant under $\Gamma_{0}(N)$, satisfies $\Delta_{0}G(z) = k$, and is square-integrable on $X_{0}(N)$ with respect to $(f,g) = \int_{\Gamma_{0}(N)\backslash \H}f(z)\overline{g(z)}\frac{dx dy}{y^{2}}$. Using the fact that the Laplace operator is self-adjoint on $L^{2}(X_{0}(N))$, we obtain
	\[
	k^{2}\text{vol}(X_{0}(N)) = (k,k) = (\Delta_{0} G,k) = (G,\Delta_{0} k) = 0,
	\]
	which implies $k = 0$. Hence $G$ is harmonic without any singularities on $X_{0}(N)$. The maximum principle for harmonic functions on $X_{0}(N)$ implies that $G$ is a constant. Thus $\Psi(f,z)/F(z)$ is a holomorphic function with constant modulus, therefore a constant. This finishes the proof of Theorem~\ref{weak converse theorem}.

\section{Proof of Theorem \ref{finiteness theorem}}\label{section finiteness theorem}

We first prove some general lemmas about newforms in $S_{2-k,\overline{\rho}_{N}}^{\new}$ and harmonic Maass forms in $H_{k,\rho_{N}}^{!}$ for $k \in \frac{1}{2} + \Z$.

\begin{Lemma}\label{lemmainnerproducts}
	Let $g \in S_{2-k,\overline{\rho}_{N/m}}^{\new}$ with $m \mid N$ be a newform with algebraic coefficients. Then the Petersson inner products $(g|U_{d}V_{\ell},g|U_{d'}V_{\ell'})$ are algebraic multiples of $(g,g)$ for all positive integers $d,\ell,d',\ell'$ with $d^{2}\ell = d^{'2}\ell' = m$.
\end{Lemma}

\begin{proof}
	Since $g$ is a newform with algebraic coefficients, Lemma~7.3 in \cite{bruinierono} tells us that there exists a harmonic Maass form $f \in H_{k,\rho_{N}}^{+}$ with algebraic principal part such that $\xi_{k}f = \|g\|^{-2}g$. We obtain by Stokes' theorem (compare \cite{bruinierfunke04}, Proposition~3.5)
	\begin{align*}
	(g|U_{d}V_{\ell},g|U_{d'}V_{\ell'}) &= \|g\|^{2}\ell^{'1-k}(g|U_{d}V_{\ell},\xi_{k}(f|U_{d'}V_{\ell'}))  \\
	&= \|g\|^{2}\ell^{'1-k}\sum_{\gamma \in D_{N}}\sum_{n \leq 0}a_{g|U_{d}V_{\ell}}(-n,\gamma)a_{f|U_{d'}V_{\ell'}}^{+}(n,\gamma).
	\end{align*}
	From the explicit action of $U_{d}$ and $V_{\ell}$ given in Section \ref{section operators} and the fact that the principal part of $f$ and the coefficients of $g$ are algebraic, we see that the right-hand is an algebraic multiple of $\|g\|^{2} = (g,g)$.
\end{proof}

\begin{Lemma}\label{lemmaorthogonalbasis}
	Every simultaneous eigenspace of the Hecke operators $T_{n}$ with $(n,N) = 1$ in $S_{2-k,\overline{\rho}_{N}}$ has an orthogonal basis consisting of cusp forms with algebraic Fourier coefficients.
\end{Lemma}

\begin{proof}
	By the direct sum decomposition \eqref{eq direct sum decomposition} of $S_{2-k,\overline{\rho}_{N}}$ and the multiplicity one theorem, every eigenspace has a basis consisting of the forms $g|U_{d}V_{\ell}$ for some newform $g \in S_{2-k,\overline{\rho}_{N/m}}^{\new}$ with $m \mid N$, where $\ell$ and $d$ run through the positive integers with $d^{2}\ell = m$. If we normalize $g$ to have algebraic coefficients, then the functions $g|U_{d}V_{\ell}$ have algebraic coefficients as well. We apply the Gram-Schmidt orthogonalization procedure (without normalization) to this basis. Using the last lemma, it follows by induction that the cusp forms in the resulting orthogonal basis are algebraic linear combinations of the $g|U_{d}V_{\ell}$ and hence have algebraic coefficients.
\end{proof}

\begin{Lemma}\label{lemmaxipreimage}
	Let $g \in S_{2-k,\overline{\rho}_{N}}$ be a cusp form with algebraic Fourier coefficients and assume that $g$ is an eigenform of all Hecke operators $T_{n}$ with $(n,N) = 1$. Then there exists a harmonic Maass form $f \in H_{k,\rho_{N}}^{+}$ with algebraic principal part such that
	\[
	\xi_{k}f = \|g\|^{-2}g.
	\]
\end{Lemma}

\begin{Remark}
	This result generalizes Lemma~7.3 in \cite{bruinierono}, where $g$ was assumed to be a newform.
\end{Remark}

\begin{proof}
	By the direct sum decomposition \eqref{eq direct sum decomposition} of $S_{2-k,\overline{\rho}_{N}}$ and the multiplicity one theorem, the assumption that $g$ is a simultaneous Hecke eigenform implies that there exists a newform $g' \in S_{2-k,\overline{\rho}_{N/m}}^{\new}$ (normalized to have algebraic coefficients) for some $m \mid N$ such that
	\[
	g = \sum_{\substack{d,\ell > 0 \\ d^{2}\ell = m}} a_{d,\ell}g'|U_{d}V_{\ell}
	\]
	with suitable coefficients $a_{d,\ell} \in \C$. Since the functions $g'|U_{d}V_{\ell}$ are linearly independent and have algebraic coefficients, and $g$ has algebraic coefficients by assumption, the $a_{d,\ell}$ are algebraic. Lemma~7.3 in \cite{bruinierono} yields some $f' \in H_{k,\rho_{N/m}}^{+}$ with algebraic principal part and $\xi_{k}f' = \|g'\|^{-2}g'$. Using Lemma \ref{lemmainnerproducts} above we obtain that $\|g\|^{-2}\|g'\|^{2}$ is algebraic.
	In particular, the harmonic Maass form
	\[
	f := \|g\|^{-2}\|g'\|^{2}\ell^{1-k}\sum_{\substack{d,\ell > 0 \\ d^{2}\ell = m}} \overline{a}_{d,\ell}f'|U_{d}V_{\ell}
	\]
	has algebraic principal part and is mapped to $\|g\|^{-2}g$ under $\xi_{k}$.
\end{proof}

\begin{Lemma}\label{lemma basis splitting}
	Let $f \in H_{k,\rho_{N}}^{+}$ be a harmonic Maass form with algebraic principal part. Write
	\[
	\xi_{k}f = \sum_{i = 1}^{m}g_{i}
	\]
	with cusp forms $g_{i} \in S_{2-k,\overline{\rho}_{N}}$ which are eigenforms under all Hecke operators $T_{n}$ with $(n,N) =1$ and which lie in distinct eigenspaces. Then we can find harmonic Maass forms $f_{i} \in H_{k,\rho_{N}}^{+}$ with algebraic principal parts such that $\xi_{k}f_{i} = g_{i}$ and
	\[
	f = \sum_{i=1}^{m}f_{i}.
	\]
\end{Lemma}

\begin{proof}
	For fixed $i$ we let $\{g_{i,j}\}_{j = 1,\dots,m_{i}} \subset S_{2-k,\overline{\rho}_{N}}$ be an orthogonal basis with algebraic coefficients of the eigenspace in which $g_{i}$ lies, compare Lemma \ref{lemmaorthogonalbasis} above. Note that the functions $g_{i,j}$ are pairwise orthogonal for all $i,j$. We can write
	\[
	\xi_{k}f = \sum_{i=1}^{m}g_{i} = \sum_{i=1}^{m}\sum_{j=1}^{m_{i}}a_{i,j}g_{i,j}
	\]
	with some $a_{i,j} \in \C$. Using Stokes' theorem (see \cite{bruinierfunke04}, Proposition~3.5) we see that
	\[
	\overline{a}_{i,j}\|g_{i,j}\|^{2} = (g_{i,j},\xi_{k}f)  = \sum_{\gamma \in D_{N}}\sum_{n \leq 0}a_{g_{i,j}}(-n,\gamma)a_{f}^{+}(n,\gamma).
	\]
	The right-hand side is algebraic, hence $\overline{a}_{i,j}\|g_{i,j}\|^{2}$ is algebraic as well. By Lemma~\ref{lemmaxipreimage} we can find harmonic Maass forms $f_{i,j} \in H_{k,\rho_{N}}^{+}$ with algebraic principal parts such that $\xi_{k}f_{i,j} = \|g_{i,j}\|^{-2}g_{i,j}$. We obtain
	\[
	f = \sum_{i=1}^{m}\sum_{j=1}^{m_{i}}\overline{a}_{i,j}\|g_{i,j}\|^{2}f_{i,j} + \tilde{f}
	\]
	for some weakly holomorphic modular form $\tilde{f} \in M_{k,\rho_{N}}^{!}$. Since $f$ and all $f_{i,j}$ have algebraic principal parts and the values $\overline{a}_{i,j}\|g_{i,j}\|^{2}$ are algebraic, the principal part of $\tilde{f}$ is algebraic as well. Then the harmonic Maass forms
	\[
	f_{i} := \sum_{j=1}^{m_{i}}\overline{a}_{i,j}\|g_{i,j}\|^{2}f_{i,j}+\frac{1}{m}\tilde{f}
	\]
	have algebraic principal parts, map to $g_{i}$ under $\xi_{k}$, and $f = \sum_{i=1}^{m}f_{i}$. This finishes the proof.
\end{proof}

\begin{Proposition}
	In the situation of Lemma~\ref{lemma basis splitting}, assume that $k = 1/2$, and that $f$ and the $f_{i}$ are orthogonal to cusp forms. Then the coefficients $a_{f}^{+}(n^{2},n)$ are algebraic for all $n \in \Z$ if and only if the coefficients $a_{f_{i}}^{+}(n^{2},n)$ are algebraic for all $n \in \Z$ and all $i$.
\end{Proposition}

\begin{proof}
	Let us assume that the coefficients $a_{f}^{+}(n^{2},n)$ are algebraic for all $n \in \Z$. By the last lemma, we know that
	\[
	f = \sum_{i=1}^{m}f_{i}.
	\]
	We apply the Hecke operator $T_{p}$ for every prime $p$ with $(p,N) = 1$, and obtain an equation of the form
	\[
	f|T_{p} = \sum_{i=1}^{m}p^{-1}\lambda_{p}(g_{i})f_{i} + \tilde{f}
	\]
	for some weakly holomorpic modular form $\tilde{f} \in M_{1/2,\rho_{N}}^{!}$, where $\lambda_{p}(g_{i})$ is the Hecke eigenvalue of $\xi_{1/2}f_{i} = g_{i}$. From the fact that $f$ (hence also $f|T_{p}$) and all $f_{i}$ are orthogonal to cusp forms and have algebraic principal parts, and the $\lambda_{p}(g_{i})$ are algebraic, we see that $\tilde{f}$ is orthogonal to cusp forms and has algebraic principal part as well. Hence $\tilde{f}$ has algebraic coefficients by Proposition~\ref{proposition rational basis}. Thus we obtain for every $p$ with $(p,N) = 1$ an equation of the form
	\[
	f^{(p)} = \sum_{i=1}^{m}\lambda_{p}(g_{i})f_{i}
	\]
	where $f^{(p)} = p (f|T_{p}-\tilde{f})$ is a harmonic Maass whose coefficients $a_{f^{(p)}}^{+}(n^{2},n)$ are algebraic for all $n \in \Z$. For fixed $n \in \Z$, we consider the linear system obtained by looking at the coefficient of the holomorphic part of index $(n^{2},n)$ in all of these equations, where we view the values $a_{f_{i}}^{+}(n^{2},n)$ with $1 \leq i \leq m$ as variables. The system has algebraic coefficients and is solvable, hence it has an algebraic solution. Since the $g_{i}$ lie in different Hecke eigenspaces, the solution is unique, therefore algebraic. This shows that $a_{f_{i}}^{+}(n^{2},n)$ is algebraic for all $n \in \Z$. The converse direction of the proposition is clear.
\end{proof}

\textit{Proof of Theorem \ref{finiteness theorem}.} We prove that item (4) implies item (5) in Theorem~\ref{finiteness theorem}. The converse implication is similar but simpler. Therefore, let us suppose from now on that the coefficients $a_{f}^{+}(n^{2},n)$ of $f$ are algebraic for all $n \in \Z$.  By the last proposition we can assume without loss of generality that $\xi_{1/2}f$ is a simultaneous Hecke eigenform. Hence there exists a newform $g \in S_{3/2,\overline{\rho}_{N/m}}^{\new}$ for some $m \mid N$ such that
\[
\xi_{1/2}f = \sum_{\substack{d,\ell > 0 \\ d^{2}\ell = m}}a_{d,\ell}g|U_{d}V_{\ell} 
\]
for some $a_{d,\ell} \in \C$. By Theorem~\ref{theorem bruinierono2} we find some $f_{g} \in H_{1/2,\rho_{N/m}}^{+}$ with principal part in the number field generated by the eigenvalues of $g$ such that $\xi_{1/2}f_{g} = \|g\|^{-2}g$, and
\begin{align}\label{eq splitting f}
f = \sum_{\substack{d,\ell > 0 \\ d^{2}\ell = m}}\overline{a}_{d,\ell}\|g\|^{2}\sqrt{\ell}f_{g}|U_{d}V_{\ell} + \tilde{f},
\end{align}
with some $\tilde{f} \in M_{1/2,\rho_{N}}^{!}$. By subtracting a suitable linear combination of unary theta functions of weight $1/2$ we can assume that $f_{g}$ is orthogonal to cusp forms and satisfies $a_{f_{g}}^{+}(0,0) = 0$. As in the proofs of the lemmas above we see that the values $\overline{a}_{d,\ell}\|g\|^{2}\sqrt{\ell}$ and all the coefficients of $\tilde{f}$ are algebraic. Let $d_{0}$ be the smallest positive integer with $d_{0}^{2}\mid m$ and $a_{d_{0},\ell_{0}} \neq 0$. We can assume that such a $d_{0}$ exists since otherwise $\xi_{1/2}f = 0$ and item (5) is trivially satisfied. Note that the Fourier expansion of $f_{g}|U_{d}V_{\ell}$ for $d > d_{0}$ is supported on indices $(n,\gamma)$ with $d^{2} \mid n$ and $d \mid \gamma$. In particular, the coefficient of index $(d_{0}^{2},d_{0})$ of the holomorphic part of
\[
\sum_{\substack{d,\ell > 0 \\ d^{2}\ell = m}}\overline{a}_{d,\ell}\|g\|^{2}\sqrt{\ell}f_{g}|U_{d}V_{\ell}
\]
equals the corresponding coefficient of
\[
\overline{a}_{d_{0},\ell_{0}}\|g\|^{2}\sqrt{\ell_{0}}f_{g}|U_{d_{0}}V_{\ell_{0}},
\]
which is algebraic by equation \eqref{eq splitting f} and the algebraicity of the correpsonding coefficient of $f-\tilde{f}$. Since $\overline{a}_{d_{0},\ell_{0}}\|g\|^{2}\sqrt{\ell_{0}} \neq 0$ is algebraic, we find that
\[
a_{f_{g}|U_{d_{0}}V_{\ell_{0}}}^{+}(d_{0}^{2},d_{0}) = a_{f_{g}}^{+}(1,1)
\]
is algebraic as well, where the equation follows from the explicit actions of $U_{d}$ and $V_{\ell}$. By Theorem~\ref{theorem bruinierono2} this implies that $L'(G,1) = 0$ for the newform $G \in S_{2}^{\new}(N/m)$ corresponding to $g$ under the Shimura correspondence. This concludes the proof of Theorem~\ref{finiteness theorem}.

\bibliography{references}{}
\bibliographystyle{plain}

\end{document}